\pdfoutput=1
\documentclass[11pt]{amsart}
\usepackage[english]{babel}
\usepackage[T1]{fontenc}
\usepackage[top=0.8in,bottom=1.1in,right=1.1in,left=1in,headheight=70pt,
headsep=0.5cm]{geometry}
\usepackage{graphicx}
 
\usepackage{color}

\usepackage{amsmath,amscd,amsthm,amsfonts,latexsym,epsfig}

\theoremstyle{plain}
\newtheorem{theorem}                 {Theorem}      [section]
\newtheorem{proposition}  [theorem]  {Proposition}

\newtheorem*{theorem*}{Theorem}
\theoremstyle{definition}
\newtheorem{example}      [theorem]  {Example}
\newtheorem{remark}       [theorem]  {Remark}
\newtheorem{definition}   [theorem]  {Definition}

\renewcommand{\Re}{\mathcal Re}
\renewcommand{\Im}{\mathcal Im}
\renewcommand{\r}{\mathbb R}
\newcommand{\C}{\mathbb C}
\newcommand{\h}{\mathbb H}

\renewcommand{\L}{\mathbb{L}}
\newcommand{\K}{\mathbb{K}}

\begin{document}
\title{Enneper representation of minimal surfaces in the three-dimensional Lorentz-Minkowski space}

\author{Irene I. Onnis}
\address{Departamento de Matem\'{a}tica\\ ICMC/USP-Campus de S\~ao Carlos\\
Caixa Postal 668\\ 13560-970 S\~ao Carlos, SP, Brazil}
\email{onnis@icmc.usp.br}

\author{Adriana A. Cintra}
     \address{Departamento de Matem\'{a}tica, C.P. 03, UFG, 75801-615, Jata\'{i}, GO, Brazil}
\email{adriana.cintra@ufg.br}
          

\thanks{The second author was supported by grant 2015/00692-5, S\~ao Paulo Research Foundation (Fapesp).}

\begin{abstract}
In this paper, we will give an Enneper-type representation for spacelike and timelike minimal surfaces in the Lorentz-Minkowski space $\L^{3}$, using the complex and the paracomplex analysis (respectively). Then, we exhibit various examples of minimal surfaces in $\L^{3}$ constructed via the Enneper representation formula, that it is equivalent to the Weierstrass representation obtained by Kobayashi (for spacelike immersions) and by Konderak (for the timelike ones).
\end{abstract}

\maketitle

\section{Introduction}

The  Weierstrass representation formula for minimal surfaces in
$\r^3$ is a powerful tool to construct examples and to
prove general properties of such surfaces, since it gives a
parametrization of minimal surfaces by holomorphic data. In \cite{MMP} the authors describe
a general Weierstrass representation formula for simply connected immersed minimal surfaces in an arbitrary Riemannian manifold. The partial differential equations
involved are, in general, too complicated to find explicit solutions.
However, for particular ambient $3$-manifolds, such as the Heisenberg
group, the hyperbolic space and the product of the hyperbolic plane
with $\r$, the equations become simpler and the formula can be
used to construct examples of conformal minimal immersions (see \cite{K}, \cite{MMP}).

In \cite{a},  Andrade introduces a new method to obtain minimal surfaces in the Euclidean $3$-space which is equivalent to the classical Weierstrass representation and, also, he proves that any immersed minimal surfaces in $\r^{3}$ can be obtained using it.
This method has the advantage of computational simplicity, with respect to the Weierstrass representation  formula, and it allows to construct a conformal minimal immersion $\psi:\Omega\subset \C\to \C\times\r$, from a harmonic function $h:\Omega\to\r$,  provided that we choose holomorphic complex valued functions $L,P$ on the simply connected domain $\Omega$ such that $L_z\,P_z=(h_z)^{2}.$ The immersion results in $\psi(z)=(L(z)-\overline{P(z)},h(z))$ and it is called {\it Enneper immersion} associated to $h$. Besides, the image $\psi (\Omega)$ is called an  {\it Enneper graph} of $h$. 

Some extensions of the Enneper-type representation in others ambient spaces have been given in  \cite{benoit} and  \cite{MO}.
The aim of this paper is to discuss an Enneper-type representation for minimal surfaces
in the Lorentz-Minkowski space $\L^3$, i.e. the affine three space $\mathbb{R}^3$ endowed with the Lorentzian metric
$$g=dx_1^2+dx_2^2-dx_3^2.$$

In the space $\L^{3}$ a Weierstrass representation type theorem was proved by Kobayashi for  spacelike minimal immersions (see \cite{Kob}),
and by Konderak  for the case of timelike minimal surfaces (see \cite{konderak}). The results of Konderak have been generalized by Lawn in \cite{l}.
Recently, these theorems were extended for immersed minimal surfaces in Riemannian and Lorentzian three-dimensional manifolds
 by Lira et al. (see \cite{Liramm}).

The paper is organized as follows. In Section~\ref{algL} we recall some basics facts of Lorentzian calculus, which plays the role of complex calculus in the classical case, for timelike minimal surfaces. Section~\ref{weier} is devoted to present a Weierstrass type representation for minimal surfaces in the three-dimensional Lorentz-Minkowski  space. We will treat the cases of spacelike and timelike minimal surfaces (given in \cite{Kob} and \cite{konderak}, respectively) in an unified approach (see Theorem~\ref{teo1}). In the Sections~\ref{four} and \ref{five} we give an Enneper-type representation for spacelike and timelike minimal surfaces in $\L^3$, using the complex and the paracomplex analysis, respectively (see Theorems~\ref{teo-spacelike} and \ref{teo-timelike}). Besides, we show that any spacelike (respectively, timelike) minimal surface in $\L^3$ can be, locally, rendered as the Enneper graph of a real valued harmonic function defined on a (para)complex domain (see Theorems~\ref{rend1} and \ref{rend2}). In addition, we use these results to provide a description of the spacelike (respectively, timelike) helicoids and catenoids given in \cite{ACM,Kob} in terms of their (para)complex Enneper data.
Finally, in Section~\ref{final} we use the Enneper-type representation to construct new interesting examples of minimal surfaces in $\L^3$ and, also, we explain how to produce new minimal surfaces starting from the Enneper data of known minimal surfaces.

\section{The algebra $\L$ of the paracomplex numbers}\label{algL}

In \cite{konderak}, the author uses paracomplex analysis to prove a Weierstrass representation formula for timelike minimal surfaces immersed in the space $\L^3$.
We recall that the algebra of {\em paracomplex (or Lorentz) numbers} is the algebra $$\L = \{a + \tau\, b\;|\; a,b \in \r\},$$ where $\tau$ is an imaginary
unit with $\tau^2 = 1$. The two internal operations are the obvious ones. We define the conjugation in $\L$ as $\overline{a + \tau\, b} := a - \tau\, b$ and the
$\L$-norm of $z = a + \tau\, b \in \L$ is defined by $|z| = |z\,\overline{z}|^{\frac{1}{2}} = |a^2 - b^2|^{\frac{1}{2}}.$
The algebra $\L$ admits the set of zero divisors $K = \{a\pm \tau\,a \,:\, a~\neq~0\}$.
If $z\notin K\cup\{0\}$, then it is invertible with inverse $\displaystyle z^{-1} = \bar{z}/(z\bar{z})$.
We observe that $\L$ is isomorphic to the algebra $\r \oplus \r$ via the map
$
\Phi(a + \tau\, b) = (a + b, a - b)
$
and the inverse of this isomorphism is given by $\Phi^{-1}(a,b)=(1/2)\,[(a+b)+\tau (a-b)]$.
Also, $\L$ can be canonically endowed with an indefinite metric by
$$\langle z,w\rangle= \Re\, (z\,\bar{w}), \qquad z,w\in \L.$$
In the following, we introduce the notion of the differentiability
over Lorentz numbers and some properties (look \cite{konderak2} for more details). 
\begin{definition}
Let $\Omega\subseteq \L$ be an open set\footnote{ The set $\L$ has a natural topology since it's a two dimensional real vector space.} and $z_0 \in \Omega$.
The $\L$-derivative of a function $f:\Omega\rightarrow\L$ at $z_0$ is defined by
\begin{equation}\label{eq:A11}
f'(z_0):= \lim_{z\rightarrow z_{0} \atop{z - z_{0} \in \L\setminus K\cup\{0\}}}\frac{f(z) - f(z_0)}{z - z_0},
\end{equation}
if the limit exists. If $f'(z_0)$ exists, we will say that $f$ is $\L$-differentiable
at $z_0$. When $f$ is $\L$-differentiable at all points of $\Omega$ we say that $f$ is $\L$-holomorphic in $\Omega$.
\end{definition}

\begin{remark}
The condition of $\L$-differentiability is much less restrictive that the usual complex differentiability.
For example, $\L$-differentiability at $z_0$ does not imply continuity at $z_0$. However, $\L$-differentiability
in an open set $\Omega \subset \L$ implies usual differentiability in $\Omega$. Also, we point out that there exist $\L$-differentiable functions of any class of usual differentiability (see \cite{konderak2}).
\end{remark}
 Introducing the paracomplex operators:
\begin{equation}\label{dpc}
\frac{\partial}{\partial z} = \frac{1}{2}\Big(\frac{\partial}{\partial u} + \tau\frac{\partial}{\partial v}\Big),\qquad
\frac{\partial}{\partial \bar{z}} = \frac{1}{2}\Big(\frac{\partial}{\partial u} - \tau\frac{\partial}{\partial v}\Big),\nonumber
\end{equation}
where $z = u + \tau\, v$, we can give a necessary and sufficient condition for the $\L$-differentiability of a function $f$ in some open set.
\begin{theorem} Let $a,b:\Omega\to\L$ be $C^1$ functions in an open set $\Omega\subset\L$. Then the function $f(u,v) = a(u,v) +\tau\, b(u,v)$, $u + \tau v\in\Omega$, is $\L$-holomorphic in $\Omega$ if and only if \begin{equation}\label{eqldif}
\displaystyle\frac{\partial f}{\partial \bar{z}} = 0
\end{equation}
is satisfied at all point of $\Omega$.
\end{theorem}
Observe that the condition~\eqref{eqldif} is equivalent to 
the para-Cauchy-Riemann equations
$$
a_u = b_v,\quad
a_v =b_u
$$
and, in this case, we have that
$$\begin{aligned}f'(z)&=a_u(u,v)+\tau\,b_u(u,v)=b_v(u,v)+\tau\,a_v(u,v)\\&=\frac{1}{2}\Big(\frac{\partial}{\partial u} + \tau\frac{\partial}{\partial v}\Big)(f).\end{aligned}$$
\begin{remark}
If $f$ is a  $\L$-differentiable function, from the  para-Cauchy-Riemann equations we have that
\begin{equation}\label{deri}
f_z=2 (\Re f)_z=2\tau (\Im f)_z.
\end{equation}
\end{remark}
We finish this part by considering the following result which will be useful later. 
\begin{proposition}\label{integral}
Let $h:\Omega\to \r$ be a function defined in the simply connected open set  $\Omega\subset\K$. Then,
$$h(z)-h(z_0)=2\,\Re \int_{z_0}^{z} h_z(z)\,dz,$$
where the integration is performed in paths contained in $\Omega$ from $z_0$ to $z$.
\end{proposition}
\begin{proof}
As $h$ is a real valued function, we have that
$\int_{z_0}^{z} h_{\bar{z}}\,d\bar{z}=\overline{\int_{z_0}^{z} h_z\,dz}.$
Therefore, $$h(z)-h(z_0)=\int_{z_0}^{z} h_z\,dz+\int_{z_0}^{z} h_{\bar{z}}\,d\bar{z}=2\,\Re \int_{z_0}^{z} h_z(z)\,dz.
$$
\end{proof}
\subsection{Some elementary functions over the Lorentz numbers}\label{formulas}
In the following, we shall write functions of the Lorentz variable $z=u +\tau v$ in the ``sans serif style'' to distinguish theme from the respective complex classical functions whose domain is contained in $\C$. In \cite{konderak2} the authors define the exponential function  
$$\mathsf{exp}(z):=e^u\,(\cosh v+\tau\,\sinh v),\quad z\in\L.$$ Putting $u=0$, we obtain 
$$\mathsf{exp}(\tau\,v)=\cosh v+\tau\,\sinh v, \quad \mathsf{exp}(-\tau\,v)=\cosh v-\tau\,\sinh v$$ and
$$
\cosh v=\frac{\mathsf{exp}(\tau v)+\mathsf{exp}(-\tau v)}{2},\quad \sinh v=\frac{\mathsf{exp}(\tau v)-\mathsf{exp}(-\tau v)}{2\tau}.
$$
These expressions may be used to continue hyperbolic cosine and sine as $\L$-holomorphic functions in
the whole set $\L$ setting
$$
\mathsf{cosh} (z):=\frac{\mathsf{exp}(\tau z)+\mathsf{exp}(-\tau z)}{2},\quad
\mathsf{sinh} (z):=\frac{\mathsf{exp}(\tau z)-\mathsf{exp}(-\tau z)}{2\tau },
$$
for all $z\in \L$. It's easy to check the following formulas
\begin{equation}
\begin{aligned}
\mathsf{cosh} (z)&=\cosh u\,\cosh v+\tau\, \sinh u\,\sinh v,\\
\mathsf{sinh} (z)&=\sinh u\,\cosh v+\tau\,\cosh u\,\sinh v.
\end{aligned}
\end{equation}
We observe that $$\mathsf{exp}(\tau\,z)=\mathsf{cosh} (z)+\tau\,\mathsf{sinh} (z)$$ and $\mathsf{sinh}' (z)=\mathsf{cosh} (z),$ $\mathsf{cosh}'(z)=\mathsf{sinh} (z)$, for all $z\in\L.$
Also, 
\begin{equation}\label{proprieta}
\mathsf{cosh} (\tau z)=\mathsf{cosh} (z), \quad \mathsf{sinh} (\tau z)=\tau\,\mathsf{sinh} (z),\quad z\in\L.
\end{equation}
Extending \eqref{proprieta} to circular trigonometric functions and applying the usual angle addition formulas, we define
\begin{equation}
\begin{aligned}
\mathsf{sin} (z)&:=\sin u\,\cos v+\tau\,\cos u\,\sin v,\\
\mathsf{cos} (z)&:=\cos u\,\cos v-\tau\, \sin u\,\sin v,\quad z\in\L.
\end{aligned}
\end{equation}
These functions are $\L$-differentiables in $\L$ and they satisfy the same differentiation formulas which hold
for real and complex variables.

\section{The Weierstrass representation formula in $\L^3$}\label{weier}
We will denote by $\K$ either the complex numbers $\C$ or the paracomplex numbers $\L$, and by $\Omega\subset \K$  an open set. Given  a smooth immersion $\psi:\Omega\subset \K \rightarrow \L^3$, we endow $\Omega$ with the induced metric $ds^{2}=\psi^{*} g$, that makes $\psi$ an isometric immersion. We will say that $\psi$ is {\it spacelike} if  $ds^{2}$ is a Riemannian metric, and that
$\psi$ is {\it timelike} if the induced metric is a Lorentzian metric.

We observe that in the Lorentzian case, we can endow $\Omega$ with paracomplex isothermic coordinates and, as in the Riemannian case, they are locally described by paracomplex isothermic charts with conformal changes of coordinates (see \cite{tila}). Let $z=u+i\,v$ (respectively, $z=u+\tau\,v$) be a complex (respectively, paracomplex) isothermal coordinate in $\Omega$, 
so that
$$ds^2\Big(\dfrac{\partial}{\partial u},\dfrac{\partial}{\partial u}\Big)=\varepsilon \,ds^2\Big(\dfrac{\partial}{\partial v},\dfrac{\partial}{\partial v}\Big),\qquad ds^2\Big(\dfrac{\partial}{\partial u},\dfrac{\partial}{\partial v}\Big)=0,$$
where $\varepsilon=1$ (respectively, $\varepsilon=-1$).
It follows that 
there exists a positive function $\lambda:\Omega\to\r$ such that the induced metric is given by $ds^{2}=\lambda\,(du^{2}+\varepsilon\, dv^{2})$, where
 \begin{equation}\label{lambda}
\lambda=\frac{g(\psi_u,\psi_u)+\varepsilon g(\psi_v,\psi_v)}{2}=2\,g( \psi_z,\psi_{\bar{z}}).
\end{equation}
Observe that the Beltrami-Laplace operator (with respect to $ds^{2}$) is given by:
\begin{equation}\label{beltrami}
\triangle=\frac{1}{\lambda}\,\Big(\dfrac{\partial}{\partial u}\,\dfrac{\partial}{\partial u}+\varepsilon\, \dfrac{\partial}{\partial v}\,\dfrac{\partial}{\partial v}\Big)=\frac{4}{\lambda}\,\dfrac{\partial}{\partial \bar{z}}\,\dfrac{\partial}{\partial z}.
\end{equation}
Also, denoting by $N$ the unit normal vector field along $\psi$, which is timelike (respectively, spacelike) if $\psi$ is a spacelike (respectively, timelike) immersion (i.e. $g(N,N)=-\varepsilon$), it results that
$\triangle \psi=-\varepsilon\,\overrightarrow{H},$
where $\overrightarrow{H}=H\,N$ is the mean curvature vector of $\psi$ (i.e. the trace of the second fundamental form with respect
to the first fundamental). In particular, the immersion $\psi$ is minimal (i.e. $H\equiv 0$ ) if and only if the coordinate functions $\psi_j$, $j=1,2,3$, are harmonic functions, or equivalently  $(\partial \psi_j/\partial z)$,  $j=1,2,3$, are $\K$-differentiable.

In the following, we state the Weierstarss representation type theorem for spacelike  (respectively, timelike) minimal immersions in $\L^{3}$, that was proved by Kobayashi in \cite{Kob} (respectively, by Konderak in \cite{konderak}), in a unified version.

\begin{theorem}[Weierstrass Representation]\label{teo1}
Let $\psi:\Omega\subset \K \rightarrow \L^3$ be a smooth conformal minimal  spacelike  (respectively, timelike) immersion. Then, the (para)complex tangent vector defined by
\begin{equation}
\phi(z) := \frac{\partial \psi}{\partial z}\bigg{|}_{\psi(z)}=\sum_{i=1}^3  \phi_i\,\frac{
\partial}{\partial x_i},\nonumber
\end{equation}
 satisfy the following conditions:
\begin{itemize}
\item[(i)]$ \phi_1\,\overline{\phi_1} + \phi_2\,\overline{\phi_2}- \phi_3\,\overline{\phi_3} \neq 0,$
\item[(ii)] $\phi_1^{2}+\phi_2^{2}-\phi_3^{2}= 0,$
\item[(iii)] $\displaystyle\frac{\partial\phi_j}{\partial\bar{z}} = 0, \; j=1,2,3$,
\end{itemize}
where $\dfrac{\partial}{\partial z}$ and $ \dfrac{\partial}{\partial\bar{z}}$ are the  (para)complex operators.

Conversely, if $\Omega\subset\K$ is a simply connected domain and $\phi_j:\Omega\to \K$, $j = 1,2,3$, are (para)complex functions satisfying the
conditions above, then the map
\begin{equation}
\psi= 2\,\Re\int_{z_0}^{z} \phi\, dz,
\end{equation}
is a well-defined conformal spacelike (respectively, timelike) minimal immersion in $\L^{3}$ (here, $z_0$ is an arbitrary fixed point of $\Omega$ and the integral
is along any curve joining $z_0$ to $z$)\footnote{The $\K$-differentiability ensures that the  $1$-forms $\phi_ j\, dz$, $j=1,2,3$, don't have real periods in $\Omega$.}.
\end{theorem}

\begin{remark}
The first condition of Theorem~\ref{teo1} ensures that $\psi$ is an immersion (see \eqref{lambda}), the second one that $\psi$ is conformal and the third one that $\psi$ is minimal.
\end{remark}

\section{Enneper-type spacelike minimal immersions in $\L^{3}$}\label{four}
In this section and in the successive, we prove an Enneper-type representation formula for spacelike and timelike (respectively) minimal surfaces immersed in the three-dimensional Lorentz-Minkowski space. Our approach considers the complex numbers for the spacelike immersions, and the algebra of the Lorentz numbers (described in Section~\ref{algL}) for the timelike ones. 

We start by considering the conformal  spacelike minimal immersion given by:
$$\psi(z)=\Big(u+\frac{u^{3}}{3}-u\,v^{2}, -v-\frac{v^{3}}{3}+v\,u^{2},v^{2}-u^{2} \Big),\qquad z\in\Omega,$$ where $\Omega=\{z\in \C\;|\; |z|\neq 1\}$, called {\em Enneper immersion of 1st kind} (see \cite{K}). Writing
$$\psi(z)=\Big(\bar{z}+\frac{z^{3}}{3}, -\Re\,(z^{2})\Big),\quad z\in\Omega,$$
and putting
$$h(z)=-\Re\, (z^{2}),\qquad L(z)=\frac{z^{3}}{3},\qquad P(z)=z,\quad z\in \Omega,$$
we observe that  $L,P: \Omega\to\C$ are holomorphic functions and $h$ is a harmonic real valued function such that $(h_z)^{2}=L_z\,P_z$. Also, we have that $|L_z|-|P_z|=|z|^{2}-1\neq 0$, $z\in\Omega$.

In this context, we prove the theorem below and, also, Theorem~\ref{rend1}.
\begin{theorem}\label{teo-spacelike}
Let $h:\Omega\to\r$ be a harmonic function in the simply connected domain $\Omega\subset\C$ and $L,P:\Omega \to \C$ two holomorphic functions such that the following conditions are satisfied:
\begin{equation}\label{um1}
(h_z)^{2}=L_z\,P_z
\end{equation}
and
\begin{equation}\label{dois1}
|L_z|-|P_z|\neq 0.
\end{equation}
Then, the map $\psi: \Omega\to\C\times\r$, given by $\psi(z)=(L(z)+\overline{P(z)},h(z))$, defines a conformal spacelike minimal immersion into $\L^{3}$.
\end{theorem}
\begin{proof}
Let us consider the three complex valued functions on $\Omega$ given by:
$$\phi_1=\frac{L_z+P_z}{2},\qquad  \phi_2=\frac{i\,(P_z-L_z)}{2},\qquad \phi_3=h_z.$$
As $L_z=(\phi_1+i\,\phi_2)$ and $P_z=(\phi_1-i\,\phi_2),$
 from \eqref{um1} it results  that
$$\phi_1^{2}+\phi_2^{2}-\phi_3^{2}=(\phi_1+i\,\phi_2)\,(\phi_1-i\,\phi_2)-\phi_3^{2}=L_z\,P_z-(h_z)^{2}=0.$$
Also, from \eqref{um1} and \eqref{dois1}, we obtain that
\begin{equation}
\begin{aligned}
2\,(\phi_1\,\overline{\phi_1} + \phi_2\,\overline{\phi_2}- \phi_3\,\overline{\phi_3} )&=|L_z|^{2}+|P_z|^{2}-2 \,|h_z|^{2}\\&=(|L_z|-|P_z|)^{2}> 0.
\end{aligned}
\end{equation}
We now observe that, since $h$ is a harmonic function (i.e. $h_{uu}+h_{vv}=0$),  we have that $\phi_3$ is holomorphic (see Section~\ref{weier}). Moreover, the holomorphicity of $L$ and $P$ implies that the real and imaginary parts of $L$ and $P$ are harmonic functions and we can write $$\phi_1=\frac{\partial \Re(L+P)}{\partial z},\quad \phi_2=\frac{\partial \Im(L-P)}{\partial z}.$$ Therefore, $(\phi_1)_{\overline{z}}=0=(\phi_2)_{\overline{z}}$ and, from Theorem~\ref{teo1}, we  conclude that
$$
\begin{aligned}
\psi(z)&= 2\,\Big(\Re\int \phi_1(z)\, dz+i\,\Re\int \phi_2(z)\, dz,\Re\int \phi_3(z)\, dz\Big)\\&=(L(z)+\overline{P(z)},h(z))
\end{aligned}
$$
is a conformal spacelike minimal immersion into $\L^{3}$.
\end{proof}
In analogy to the Euclidean $3$-space (see \cite{a}), we shall call the immersion $\psi=(L+\overline{P},h)$ an {\it Enneper spacelike immersion} associated to $h$, its image an {\it Enneper graph} of $h$ and 
$\mathcal{D}_{\psi}^{\tiny{\C}}=(L_z,P_z,h_z)$ the {\it Enneper complex data} of $\psi$. 

We shall now illustrate the Theorem~\ref{teo-spacelike} with some known examples of spacelike minimal immersions in $\L^{3}$. Specifically, we will consider the natural analogues (spacelike) surfaces in $\L^{3}$ to the classical catenoid and helicoid.
\begin{example}[Spacelike catenoid of 1st kind]
Set $\Omega=\{z\in \C : |z|> 1\}$.
Let  $L,P:\Omega\to \C$ be the holomorphic functions defined by
$$L(z)=\frac{z}{2},\qquad P(z)=-\frac{ 1}{2\,z},$$
and $h(z)=\Re\,(\ln z)$, that is a harmonic function in $\Omega$. We observe that condition \eqref{um1} is satisfied and, also,
$$
|L_z|-|P_z|=\frac{|z|^2-1}{2\,|z|^2}\neq 0,\qquad z\in\Omega.$$
Then, from Theorem~\ref{teo-spacelike}, the corresponding spacelike minimal immersion is given by:
$$
\psi(z)=\Big(\frac{1}{2}\Big(z-\frac{1}{\overline{z}}\Big),\Re\,( \ln z)\Big)
$$
and it represents the {\em catenoid of 1st kind} (also called {\em elliptic catenoid}) described in \cite{K}.
Introducing polar coordinates $z=r\,e^{i\,\theta}$, we can write
$$\psi(r,\theta)=(\sinh (\ln r)\,\cos\theta,\sinh (\ln r)\,\sin\theta,\ln r),\qquad r>1,$$
so
$x_1^{2}+x_2^{2}=\sinh^{2} x_3$, with $ x_3> 0$.
\end{example}

\begin{example}[Spacelike helicoid of 1st kind]
Now, we describe the conjugate surface of the elliptic catenoid, which image in $\r^{3}$ is an open subset of the classical minimal helicoid, $x_1\,\cos x_3+x_2\,\sin x_3=0.$
For this, we consider  $\Omega=\{z\in \C : |z|> 1\}$, the holomorphic functions  
$$L(z)=\frac{i\,z}{2},\qquad P(z)=-\frac{i}{2\,z},\qquad z\in\Omega,$$
and the harmonic function  $h(z)=\Im\,( \ln z)$, defined in $\Omega$. As
$$h_z(z)=-\frac{i}{2\,z}, \qquad L_z(z)=\frac{i}{2},\qquad P_z(z)=\frac{i}{2\,z^{2}},$$
it results that $$L_z\,P_z=h_z^{2},\qquad
|L_z|-|P_z|=\frac{|z|^2-1}{2\,|z|^2}\neq 0,\qquad z\in\Omega.$$
Therefore, from Theorem \ref{teo-spacelike}, the corresponding spacelike minimal immersion is given by:
$$
\psi(z)=\Big(\frac{i}{2}\Big(\frac{1}{\overline{z}}+z\Big),\Im\,( \ln z)\Big)
$$
and it represents the {\em helicoid of 1st kind} given in \cite{K}. Using polar coordinates $z=r\,e^{i\,\theta}$, we get
$$\psi(r,\theta)=(-\cosh (\ln r)\,\sin\theta,\cosh (\ln r)\,\cos\theta,\theta),\quad r> 1.$$
\end{example}

In the next theorem, we will show that any spacelike minimal surface in the Lorentz-Minkowski $3$-space can be rendered as the Enneper graph of a harmonic function.
\begin{theorem}\label{rend1}
Let $\tilde\psi:{\mathcal M}^2\to \L^{3}\equiv\C\times\r$ be a
minimal immersion of a spacelike surface ${\mathcal M}$ in $\L^{3}$.
Then, there exists a simply connected domain $\Omega\subset\C$ and a harmonic function $h:\Omega\subset\C\to \r$
such that the immersed minimal surface $\tilde\psi({\mathcal M})$ is
an Enneper graph of $h$.
\end{theorem}
\begin{proof}
Suppose that the minimal immersion is given by
$\tilde\psi=(\tilde\psi_1+i\,\tilde\psi_2,\tilde\psi_3)$. Since ${\mathcal M}$ is a spacelike minimal surface it cannot be compact (on the contrary, $\tilde\psi$  would be a harmonic function on a compact Riemannian surface, hence constant) so, from the Koebe's Uniformization Theorem, it results that its covering space $\Omega$ is either the complex plane $\C$ or the open unit complex disc.

We denote by $\pi: \Omega\to {\mathcal M}$ the universal covering of ${\mathcal M}$ and by $\psi:\Omega\to \L^{3}$  the lift of $\tilde\psi$, i.e. $\psi=\tilde\psi\circ \pi$. As $\psi$ is a
conformal minimal immersion, it follows that
\begin{equation}\label{ro}
\begin{aligned}
0&=(\psi_1)_z^2+(\psi_2)_z^2-(\psi_3)_z^2\\&=[(\psi_1)_z+i\,(\psi_2)_z]\,[(\psi_1)_z-i\,(\psi_2)_z]-(\psi_3)_z^2
\end{aligned}
\end{equation}
and, also, $(\psi_i)_z$, $i=1,2,3$, are holomorphic.
Fixed a point $z_0\in\Omega$, the equation
{\eqref{ro}} suggests to define the following functions:
\begin{equation}\begin{aligned}L(z)&=\int_{z_0}^z
[(\psi_1)_z+i\,(\psi_2)_z]\,dz,\\
P(z)&=\int_{z_0}^z [(\psi_1)_z-i\,(\psi_2)_z]\,dz.\end{aligned}\end{equation}
Since $\Omega$ is a
simply connected domain and the integrand functions are holomorphic, the above integrals don't depend on the path from
$z_0$ to $z$, so $L$ and $P$ are well-defined  holomorphic functions.
We shall prove that
$\psi(z)=(L(z)+\overline{P(z)},h(z))$, where $h(z):=\psi_3(z)$ is a harmonic function (because $(\psi_3)_{z\overline{z}}=0$). For this,  we note that
\begin{equation}\begin{aligned}\nonumber L(z)+\overline{P(z)}&=
\int_{z_0}^z [(\psi_1)_z+i\,(\psi_2)_z]\,dz+\int_{z_0}^z
[(\psi_1)_{\overline{z}}+i\,(\psi_2)_{\overline{z}}]\,d\overline{z}\\&=
\int_{z_0}^z d\psi_1+i\,\int_{z_0}^z d\psi_2=\psi_1(z)+i\,\psi_2(z),
\end{aligned}\end{equation} where, in the last equality, we
have assumed (without loss of generality) that $\psi(z_0)=(0,0,0)$.
Besides, we observe that equation \eqref{ro} can be written as
\begin{equation}\label{tre}
L_z\,P_z-(h_z)^{2}=0,
\end{equation}
that is the condition~\eqref{um1} of Theorem \ref{teo-spacelike}.
 Finally, to prove that
$\psi$ is an Enneper immersion associated to the harmonic function
$h$, it remains to verify the equation {\eqref{dois1}}.
As
$$(\psi_1)_z=\frac{L_z+P_z}{2},\qquad (\psi_2)_z=\frac{i\,(P_z-L_z)}{2},$$ taking into account \eqref{tre}, we have that
$$0<  2\,g( \psi_z,\psi_{\overline{z}})=|L_z|^{2}+|P_z|^{2}-2 \,|h_z|^{2}=(|L_z|-|P_z|)^{2}.$$
 because of  $\psi$ is an immersion. This completes the
 proof.
\end{proof}

Using the Theorem~\ref{rend1}, we have determined the Enneper data of the spacelike catenoids and helicoids described in \cite{ACM,Kob} and we have collected them in the Tables~\ref{tab1} and \ref{tab2}, respectively. Before, we observe that in \cite{Kob} the elliptic catenoid (respectively, hyperbolic catenoid, parabolic catenoid) is called 
catenoid of first kind (respectively, catenoid of second kind, Enneper surface of second kind\footnote{In the Table~\ref{tab1} we have considered the parabolic catenoid parametrized by (see \cite{Kob}):
$$\psi(u,v)=\Big(u-uv^{2}+\frac{u^{3}}{3},-2uv,-u-uv^{2}+\frac{u^{3}}{3}\Big),\qquad u\neq 0.$$}).
In Section~\ref{new} we will use the Tables~\ref{tab1} and \ref{tab2} to  construct new interesting examples of minimal surfaces in $\L^3$.
\begin{table}[h!]
\caption{Enneper data for spacelike catenoids in $\L^{3}$.}\label{tab1}
\begin{tabular}{|c|c|c|c|c|}
 \hline
& &&&\\  $\mathbf{L_z}$ & $\mathbf{P_z}$ & $\mathbf{h_z}$ & \textbf{spacelike surface}&  \textbf{catenoid}\\ &&&&\\
\hline
 & &&&\\ $\dfrac{1}{2}$ & $\dfrac{1}{2z^{2}}$ & $-\dfrac{1}{2z}$ & $x_1^2+x_2^2=(\sinh x_3)^2$ & \text{elliptic}\\ & &&&
  \\
  \hline
  &&& &\\ $\dfrac{1+\cos z}{2}$ & $\dfrac{1-\cos z}{2}$ & $-\dfrac{\sin z}{2}$ & $x_3^2-x_2^2=(\cos x_1)^2$ & \text{hyperbolic }\\ & &&&
  \\
  \hline
   &&& &\\ $\dfrac{\cosh z -1}{2}$ & $\dfrac{\cosh z+1}{2}$ & $\dfrac{\sinh z}{2}$ & $x_3^2-x_1^2=(\cos x_2)^2$ & \text{hyperbolic }\\ & &&&
  \\
  \hline
  &&& &\\ $\dfrac{(1-z)^2}{2}$ & $\dfrac{(1+z)^{2}}{2}$ & $\dfrac{z^{2}-1}{2}$ & $12(x_1^2+x_2^2-x_3^2)=(x_1-x_3)^4$ & \text{parabolic }\\ & &&&
  \\
  \hline
\end{tabular}\end{table}
~\begin{table}[h!]\caption{Enneper data for spacelike helicoids in $\L^{3}$.}\label{tab2}
\begin{tabular}{|c|c|c|c|c|}
 \hline
& &&&\\  $\mathbf{L_z}$ & $\mathbf{P_z}$ & $\mathbf{h_z}$ & \textbf{spacelike surface}&  \textbf{helicoid} \\ &&&&\\
\hline
  &&& &\\ $\dfrac{i}{2}$ & $\dfrac{i}{2z^{2}}$ & $-\dfrac{i}{2z}$ & $x_1=-x_2\tan x_3$ & \text{ of 1st kind}\\ & &&&
  \\
  \hline
   &&& &\\ $\dfrac{\cos z+1}{2}$ & $\dfrac{\cos z-1}{2}$ & $-\dfrac{i\,\sin z}{2}$ & $x_3=x_1\tanh x_2$ & \text{ of 2nd kind} \\ & &&&
  \\
  \hline
  &&& &\\ $\dfrac{i\,(1-z)^2}{2}$ & $\dfrac{i\,(1+z)^{2}}{2}$ & $\dfrac{i\,(z^{2}-1)}{2}$ &  $x_2=\dfrac{(x_1-x_3)^2}{6}+\dfrac{x_3+x_1}{x_3-x_1}$ & \text{parabolic}\\ & &&&
  \\
  \hline
\end{tabular}\end{table}

\section{Enneper-type timelike minimal immersions in $\L^{3}$}\label{five}
Now let us estabilish the analogue result to Theorem~\ref{teo-spacelike} for timelike minimal immersions in $\L^3$. We start considering the Lorentzian Enneper immersion given by Konderak in \cite{konderak}:
$$\psi(z)=\Big(u^{2}+v^{2}, u-\frac{u^{3}}{3}-u\,v^{2}, v+\frac{v^{3}}{3}+v\,u^{2}\Big),\qquad z\in\Omega,$$
where $\Omega=\{z\in \L\;|\; 1+z\,\overline{z}\neq 0\}$, that can be written as
$$\psi(z)=\Big(\Re\, (z^{2}), z-\frac{\overline{z}^{3}}{3}\Big),\quad z\in\Omega.$$
Observe that, putting
$$h(z)=\Re\, (z^{2}),\qquad L(z)=z,\qquad P(z)=\frac{z^{3}}{3},\quad z\in \Omega,$$
we have that  $L,P: \Omega\to\L$ are $\L$-differentiable and $h$ is a harmonic real valued function  (i.e. $h_{uu}-h_{vv}=0$) such that $(h_z)^{2}=L_z\,P_z$. Also, 
$$2h_z\,\overline{h_z}+L_z\,\overline{L_z}+P_z\,\overline{P_z}=(1+z\,\overline{z})^{2}> 0,\quad z\in\Omega.$$
In this regard we prove the following theorem.
\begin{theorem}\label{teo-timelike}
Let $h:\Omega\to\r$ be a harmonic function in the simply connected domain $\Omega\subset\L$ and $L,P:\Omega \to \L$ two $\L$-differentiable functions such that the following conditions are satisfied:
\begin{equation}\label{um}
(h_z)^{2}=L_z\,P_z
\end{equation}
and
\begin{equation}\label{dois}
2\,h_z\,\overline{h_z}+L_z\,\overline{L_z}+P_z\,\overline{P_z}\neq 0.
\end{equation}
Then, the map $\psi: \Omega\to\r\times\L$, given by $\psi(z)=(h(z),L(z)-\overline{P(z)})$, defines a conformal timelike minimal immersion into $\L^{3}$.
\end{theorem}
\begin{remark}
If  $h_z(z)\notin K\cup\{0\}$, for all $z\in\Omega$, the condition~\eqref{dois} is equivalent to $$P_z\notin K\cup\{0\}\qquad \text{and} \qquad h_z\,\overline{h_z}+L_z\,\overline{L_z}\neq 0.$$ In fact, using \eqref{um}, we can write
$$\begin{aligned}&2\,h_z\,\overline{h_z}+L_z\,\overline{L_z}+P_z\,\overline{P_z}=P_z\,\overline{P_z}\,
\Big(1+\frac{L_z\,\overline{L_z}}{h_z\,\overline{h_z}}\Big)^{2}.
\end{aligned}$$
\end{remark}
\begin{proof}
Let define three paracomplex valued functions on $\Omega$:
$$\phi_1=h_z,\qquad \phi_2=\frac{L_z-P_z}{2},\qquad  \phi_3=\frac{\tau\,(L_z+P_z)}{2}.$$
As $L_z=\phi_2+\tau\,\phi_3$ and $ P_z=-\phi_2+\tau\,\phi_3$,
from \eqref{um} and \eqref{dois}, it
results  that
$$\phi_1^{2}+\phi_2^{2}-\phi_3^{2}=\phi_1^{2}+(\phi_2+\tau\,\phi_3)\,(\phi_2-\tau\,\phi_3)=h_z^{2}-L_z\,P_z=0$$
and
$$
\begin{aligned}
2(\phi_1\,\overline{\phi_1} + \phi_2\,\overline{\phi_2}- \phi_3\,\overline{\phi_3} )=2h_z\,\overline{h_z}+L_z\,\overline{L_z}+P_z\,\overline{P_z}\neq 0.
\end{aligned}
$$
We observe that, since $h$ is a harmonic function (i.e. $h_{uu}-h_{vv}=0$),  the function $\phi_1$ is $\L$-differentiable. Moreover, the $\L$-differentiabilty of $L$ and $P$ implies that 
the real and imaginary parts of $L$ and $P$ are harmonic functions and, using \eqref{deri}, we can write $$\phi_2=\frac{\partial \Re(L-P)}{\partial z},\quad \phi_3=\frac{\partial \Im(L+P)}{\partial z}.$$
Consequently, $(\phi_2)_{\overline{z}}=0=(\phi_3)_{\overline{z}}$ and, from Theorem~\ref{teo1} and taking into account the Proposition~\ref{integral}, we conclude that
$$
\begin{aligned}
\psi(z)&= 2\,\Big(\Re\int \phi_1(z)\,dz,\Re\int \phi_2(z)\, dz+\tau\,\Re\int \phi_3(z)\, dz\Big)\\&=(h(z),L(z)-\overline{P(z)})
\end{aligned}
$$
is a conformal timelike minimal immersion into $\L^{3}$.
\end{proof}

We will call  $\psi=(h,L-\overline{P})$  an {\it Enneper timelike immersion} associated to $h$ and
$\mathcal{D}_{\psi}^{\tiny{\L}}=(L_z,P_z,h_z)$ the {\em Enneper paracomplex data} of  $\psi$.

We are going to illustrate the Theorem~\ref{teo-timelike} throught  some known examples of timelike minimal immersions into $\L^{3}$. We will use the formulas given in Section~\ref{formulas} (see \cite{konderak2}, for more details).

\begin{example}[Lorentzian catenoid]
Let $L,P:\L\to \L$ be the $\L$-differentiable functions defined by:
$$L(z)=\frac{\mathsf{cosh}\, z-\mathsf{sinh}\, z}{2},\qquad P(z)=-\frac{\mathsf{cosh}\, z+\mathsf{sinh}\, z}{2},$$
and $h(z)=u$, that is a harmonic function in $\L$. It's easy to check that $L_z\,P_z=(h_z)^{2}$ and, also,
$$
2h_z\,\overline{h_z}+L_z\,\overline{L_z}+P_z\,\overline{P_z}
=\cosh^2 u>0,\qquad z\in \L.
$$
Then, from Theorem~\ref{teo-timelike}, the corresponding timelike minimal immersion is given by
$$
\begin{aligned}
\psi(z)&=(u,\Re\, (\mathsf{cosh}\, z)-\tau\,\Im\, (\mathsf{sinh}\, z))\\
&=(u,\cosh u\,\cosh v,-\cosh u\,\sinh v),
\end{aligned}$$
and it represents the {\em Lorentzian catenoid} (see \cite{konderak}).
\end{example}

\begin{example}[Lorentzian helicoid]
In this example, we give the Enneper functions for the timelike helicoid described in \cite{konderak}. We consider in $\Omega=\{z\in\L\; |\; u\neq 0\}$ the $\L$-differentiable functions given by:
$$L(z)=\frac{\mathsf{sinh}\, z-\mathsf{cosh}\, z}{2},\qquad P(z)=\frac{\mathsf{cosh}\, z+\mathsf{sinh}\, z}{2}$$ and the harmonic function $h(z)=-v$.
As $L_z=-L$, $P_z=P$ and $h_z=-\tau/2$, the
condition~\eqref{um} is satisfied. Also,
$$
2h_z\,\overline{h_z}+L_z\,\overline{L_z}+P_z\,\overline{P_z}=\sinh^2 u> 0, \qquad z\in\Omega.
$$
Then, from Theorem~\ref{teo-timelike}, we obtain that the map
$$
\begin{aligned}
\psi(z)&=(-v,-\Re\, (\mathsf{cosh}\, z)+\tau\,\Im\, (\mathsf{sinh}\, z))\\
&=(-v,-\cosh u\,\cosh v,\cosh u\,\sinh v)
\end{aligned}$$
defines a conformal timelike minimal immersion (in a simply connected subset of $\Omega$) and it is the parametrization of the {\em Lorentzian helicoid} given in \cite{konderak}.
\end{example}

Now, we will show that any simply connected timelike minimal surfaces in the Lorentz-Minkowski $3$-space can be represented as the Enneper graph of a harmonic function. More precisely, we have the following:

\begin{theorem}\label{rend2}
Let ${\mathcal M}^2$ a timelike minimal surface in $\L^{3}$, given by the immersion $\psi:\Omega\to \L^{3}$, where $\Omega\subset\L$
is a simply connected domain. Then, there exists a harmonic function $h:\Omega\subset\L\to \r$ such that the immersed minimal surface ${\mathcal M}$ is
an Enneper graph of $h$.
\end{theorem}
\begin{proof}
In terms of proper null coordinates $x,y$ on $\Omega$ (see \cite{tila}), $ds^{2}=2Fdx\,dy$, with $F>0$, and the minimality of $\psi$ gives $m=g(\psi_{xy},N)=0.$ Therefore, as $g(\psi_{xy},\psi_x)=0=g(\psi_{xy},\psi_y)$ by $E=G=0$,
it results that $\psi_{xy}=0$. Thus, introducing in $\Omega$ the paracomplex isothermal coordinate $z=u+\tau\, v$, where $u=x+y$, $v=x-y$, we have that $\psi_{uu}-\psi_{vv}=0$ and, so,
\begin{equation}\label{compli1}
\frac{\partial (\psi_i)_z}{\partial \overline{z}}=0,\qquad i=1,2,3.
\end{equation}
The conformality of $\psi$ implies the equation
\begin{equation}
\begin{aligned}\label{ti}
0&=(\psi_1)_z^2+(\psi_2)_z^2-(\psi_3)_z^2\\&=(\psi_1)_z^2+[(\psi_2)_z+\tau\,(\psi_3)_z]\,[(\psi_2)_z-\tau\,(\psi_3)_z],
\end{aligned}
\end{equation}
that suggests to define the following functions:
\begin{equation}\begin{aligned}\label{compli2} L(z)&=\int_{z_0}^z
[(\psi_2)_z+\tau\,(\psi_3)_z]\,dz,\\
P(z)&=-\int_{z_0}^z [(\psi_2)_z-\tau\,(\psi_3)_z]\,dz,\end{aligned}\end{equation}
where $z_0\in\Omega$ is a fixed point.
Since $\Omega$ is a
simply connected domain in $\L$, the equation~\eqref{compli1} ensures that the integrals in \eqref{compli2} don't depend on the path from
$z_0$ to $z$. So $L$ and $P$ are well-defined  $\L$-holomorphic functions.
We shall prove that
$\psi(z)=(h(z),L(z)-\overline{P(z)})$, where $h(z):=\psi_1(z)$ is a harmonic function (see \eqref{compli1}). For this, we have
\begin{equation}\begin{aligned}\nonumber L(z)-\overline{P(z)}&=
\int_{z_0}^z [(\psi_2)_z+\tau\,(\psi_3)_z]\,dz+\int_{z_0}^z
[(\psi_2)_{\overline{z}}+\tau\,(\psi_3)_{\overline{z}}]\,d\overline{z}\\&=
\int_{z_0}^z d\psi_2+\tau\,\int_{z_0}^z d\psi_3=\psi_2(z)+\tau\,\psi_3(z),
\end{aligned}\end{equation} where, in the last equality, we
have assumed (without loss of generality) that $\psi(z_0)=(0,0,0)$.
Finally, we observe that \eqref{ti} can be written as
\begin{equation}\label{tre1}
L_z\,P_z-(h_z)^{2}=0,
\end{equation}
that is the condition \eqref{um} of Theorem~\ref{teo-timelike}. Therefore, using that
$$(\psi_2)_z=\frac{L_z-P_z}{2},\qquad  (\psi_3)_z=\frac{\tau\,(L_z+P_z)}{2},$$  we get 
$$0 \neq 2\,g( \psi_z,\psi_{\overline{z}})=2h_z\,\overline{h_z}+L_z\,\overline{L_z}+P_z\,\overline{P_z},$$
 because of  $\psi$ is an immersion.  This finishes the proof.
\end{proof}
Now, we will use Theorem~\ref{rend2} to provide a description of the timelike catenoids and helicoids given in \cite{CDM,konderak} in terms of their paracomplex Enneper data (see Tables~\ref{tab3} and \ref{tab4}). In the last section, we will employ these tables to determine new interesting examples of minimal surfaces in $\L^3$.
\begin{table}[h!]\caption{Enneper data for Lorentzian catenoids in $\L^{3}$.}\label{tab3}
\begin{tabular}{|c|c|c|c|c|}
 \hline
& &&&\\  $\mathbf{L_z}$ & $\mathbf{P_z}$ & $\mathbf{h_z}$ & \textbf{timelike surface}& \textbf{catenoid} \\ &&&&\\
\hline
  &&& &\\ $\dfrac{\tau\,(1+\mathsf{cos}\, z)}{2}$ & $\dfrac{\tau\,(1-\mathsf{cos}\, z)}{2}$ & $-\dfrac{\mathsf{sin} z}{2}$ & $x_1^2+x_2^2=(\cos x_3)^2$ & \text{elliptic catenoid}\\ & &&&
  \\
  \hline
  &&& &\\ $\dfrac{\mathsf{sinh}\, z-\mathsf{cosh}\, z}{2}$ & $-\dfrac{\mathsf{sinh}\, z+\mathsf{cosh} z}{2}$ & $\dfrac{1}{2}$ & $x_2^2-x_3^2=(\cosh x_1)^2$ & \text{hyp.  of 1st kind}\\ & &&&
  \\
  \hline
   &&& &\\ $\dfrac{\tau\, \mathsf{cosh}\, z +1}{2}$ & $\dfrac{\tau\, \mathsf{cosh}\, z-1}{2}$ & $\dfrac{\tau\, \mathsf{sinh\,} z}{2}$ & $x_3^2-x_1^2=(\sinh x_2)^2$ & \text{hyp. of 2nd kind}\\ & &&&
  \\
  \hline
   &&& &\\ $-\dfrac{\tau(z+1)^2}{2}$ & $-\dfrac{\tau(z-1)^2}{2}$ & $\dfrac{1- z^2}{2}$ & $12(x_3^2-x_1^2-x_2^2)=(x_1-x_3)^4$ & \text{parabolic}\\ & &&&
  \\
  \hline
\end{tabular}\end{table}
\begin{table}[h!]\caption{Enneper data for Lorentzian helicoids in $\L^{3}$.}\label{tab4}
\begin{tabular}{|c|c|c|c|c|}
 \hline
& &&&\\  $\mathbf{L_z}$ & $\mathbf{P_z}$ & $\mathbf{h_z}$ & \textbf{timelike surface}&  \textbf{helicoid} \\ &&&&\\
  \hline
   &&& &\\ $\dfrac{(1+\tau\,\mathsf{sin}\, z)}{2}$ & $\dfrac{(1-\tau\,\mathsf{sin}\, z)}{2}$ & $\dfrac{\mathsf{cos}\, z}{2}$ & $x_2=x_1\tan x_3$ & \text{of 1st kind}\\ & &&&
  \\
  \hline
    &&& &\\ $\dfrac{(\mathsf{cosh}\, z-\mathsf{sinh}\, z)}{2}$ & $\dfrac{(\mathsf{cosh}\, z+\mathsf{sinh}\, z)}{2}$ & $-\dfrac{\tau}{2}$ & $x_3=x_2\tanh x_1$ & \text{ of 2nd kind}\\ & &&&
  \\
  \hline
   &&& &\\ $\dfrac{\mathsf{cosh}\, z+\tau}{2}$ & $\dfrac{\mathsf{cosh}\, z-\tau}{2}$ & $\dfrac{\mathsf{sinh}\, z}{2}$ & $x_3=x_1\tanh x_2$ & \text{ of 2nd kind}\\ & &&&
  \\
  \hline
   &&& &\\ $\dfrac{\tau(\mathsf{cosh}\, z+\mathsf{sinh}\, z)}{2}$ & $\dfrac{\tau(\mathsf{cosh}\, z-\mathsf{sinh}\, z)}{2}$ & $\dfrac{\tau}{2}$ & $x_2=x_3\tanh x_1$ & \text{of 3rd kind}\\ & &&&
  \\
  \hline
  &&& &\\ $\dfrac{\tau(\mathsf{cosh}\, z+1)}{2}$ & $\dfrac{\tau(\mathsf{cosh}\,  z-1)}{2}$ & $\dfrac{\tau\, \mathsf{sinh}\,z}{2}$ & $x_1=x_3\tanh x_2$ & \text{ of 3rd kind}\\ & &&&
  \\
  \hline
   &&& &\\ $\dfrac{(z+1)^2}{2}$ & $\dfrac{(z-1)^2}{2}$ & $\dfrac{\tau (z^2-1)}{2}$ &  $x_2=\dfrac{(x_1-x_3)^2}{6}+\dfrac{x_3+x_1}{x_3-x_1}$ & \text{parabolic}\\ & &&&
  \\
  \hline
\end{tabular}\end{table}

\section{Construction of new minimal surfaces in $\L^{3}$}\label{final}
This section is devoted to the construction of minimal immersions in $\L^3$ starting from the Enneper data and using the Theorems~\ref{teo-spacelike} and \ref{teo-timelike}. Also, we explain how to
produce new examples of minimal surfaces starting from the Enneper data of others minimal surfaces in $\L^3$.

\subsection{Surfaces containing the involute of a circle as a pregeodesic}
First of all, we remember that a circle in $\L^3$ is the orbit of a point out of a straight line $\ell$ under a group of rotations in $\L^3$ that leave $\ell$ pointwise fixed (see \cite{LS}). Depending on the causal character of $\ell$, there are  (after an isometry of the ambient) three types of circles: Euclidean circles in planes parallel to the $x_1x_2$-plane,
Euclidean hyperbolas in planes parallel  to the $x_2x_3$-plane and
 Euclidean parabolas in planes parallel to the plane $x_2=x_3$.

\begin{example}
Let us consider the Enneper data
$${\mathcal D}_{\psi}^{\tiny{\C}}=\Big( \frac{z\,(1-\cosh z)}{2}, -\frac{z\,(1+\cosh z)}{2}, -\frac{z\,\sinh z}{2}\Big),$$ defined for all $z\in\C$, with $u\neq 0$. Applying Theorem~\ref{teo-spacelike} we obtain the spacelike minimal surface given by:
$$\begin{aligned}\psi(z)=(&\cosh u\,(\cos v+v\sin v)-u\cos v\sinh u,uv,\\ &\sinh u\,(\cos v+v\sin v)-u\cos v\cosh u).
\end{aligned}$$ We observe that this surface contains the spacelike curve
$$\psi(u,0)=(\cosh u-u\sinh u,0, \sinh u-u\cosh u), \qquad u\neq 0,$$ as a planar pregeodesic (see Figure~\ref{involute}) and, thanks to the results proved in \cite{ACM}, it's the only minimal surface in $\L^3$ which has this property. Also, the $u$-coordinate curve is the involute of the timelike circle $\alpha(u)=(\cosh u,0, \sinh u)$, with $u\neq 0$.
\end{example}

\begin{example}
Choosing the paracomplex Enneper data
$${\mathcal D}_{\psi}^{\tiny{\L}}=\Big( \frac{\tau z\,(1-\textsf{cosh} \,z)}{2}, -\frac{\tau z\,(1+\textsf{cosh}\,z)}{2}, -\frac{z\,\textsf{sinh}\, z}{2}\Big),$$ defined for all $z\in\L$, with $u\neq 0$, and using Theorem~\ref{teo-timelike}, we obtain the timelike minimal immersion given by:
$$\begin{aligned}\psi(z)=(&\cosh v\,(\sinh u-u\cosh u)-v\sinh v\sinh u,uv, \\&\cosh v\,(\cosh u-u\sinh u)-v\cosh u\sinh v).
\end{aligned}$$ This immersion is the only (see \cite{CDM}) minimal immersion in $\L^3$  that contains the timelike curve
$$\psi(u,0)=(\sinh u-u\cosh u,0, \cosh u-u\sinh u), \qquad u\neq 0,$$ as a planar pregeodesic (see Figure~\ref{involute}). This  curve is the involute of the spacelike circle $\alpha(u)=(\sinh u,0, \cosh u)$, with $u\neq 0$.
\end{example}

\begin{example}
In this example, we take  the Enneper data
$${\mathcal D}_{\psi}^{\tiny{\L}}=\Big( \frac{\,z\,(1+\tau\,\textsf{sin}\, z)}{2}, \frac{z\,(1-\tau\, \textsf{sin}\, z)}{2}, \frac{z\,\textsf{cos}\, z}{2}\Big),$$ defined for all $z\in\L$, with $v\neq 0$. From the Theorem~\ref{teo-timelike} we get the timelike minimal surface parametrized by:
$$\begin{aligned}
\psi(z)=(&\cos u\,(\cos v+v\sin v)+u\cos v\sin u,\\&\cos u\,(\sin v-v\cos v)+u\sin v\sin u,uv).
\end{aligned}$$ Note that this surface is the (only) minimal surface in $\L^3$  that contains the spacelike curve
$$\psi(0,v)=(\cos v+v\sin v,\sin v-v\cos v,0), \qquad v\neq 0,$$ as a planar pregeodesic (see Figure~\ref{involute}). This  curve is the involute of the spacelike circle $\alpha(v)=(\cos v,\sin v,0 ),$ with $v\neq 0$.
\end{example}

\begin{figure}[h!]
\begin{center}
\includegraphics[width=0.28\textwidth]{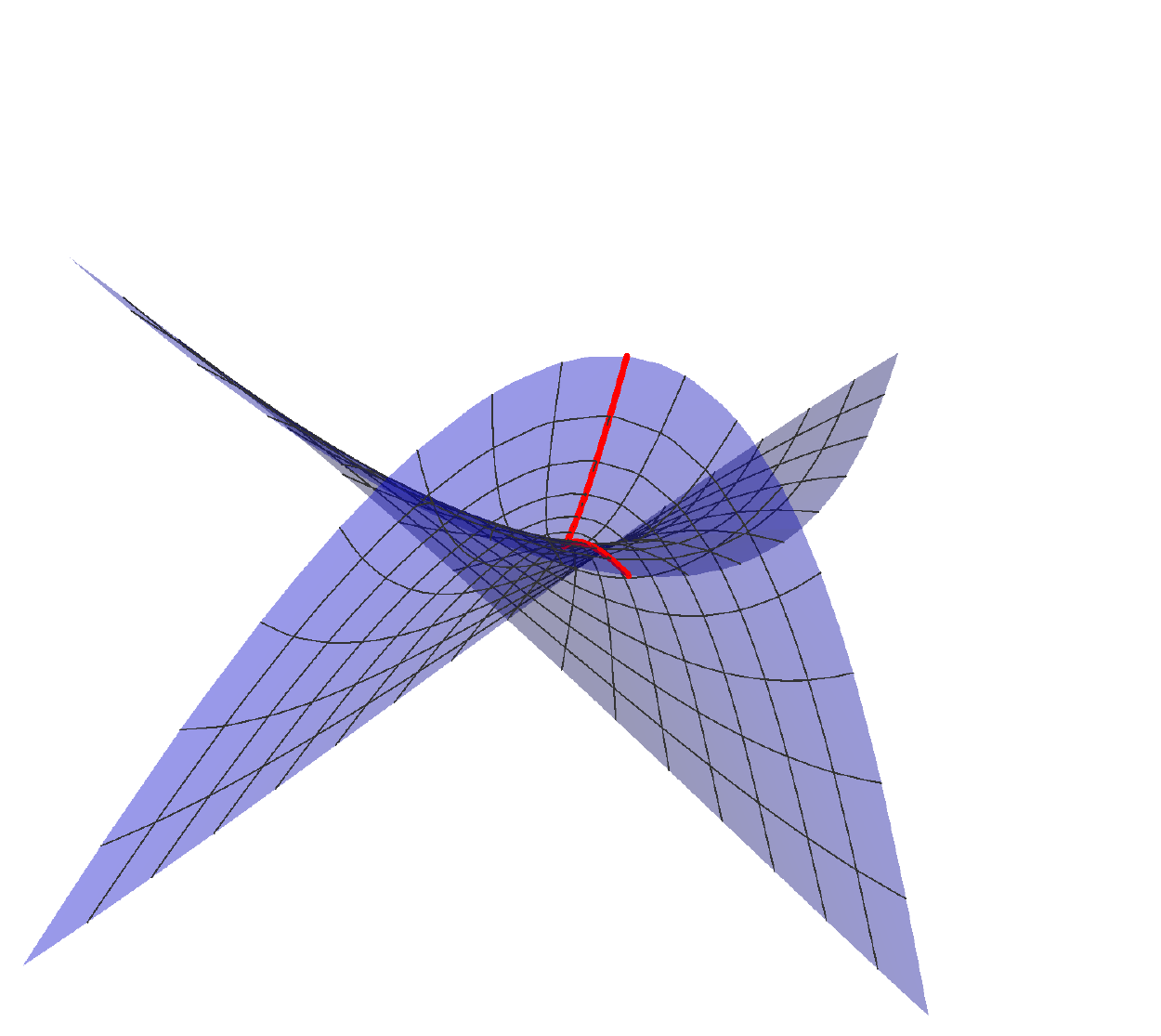}
\includegraphics[width=0.36\textwidth]{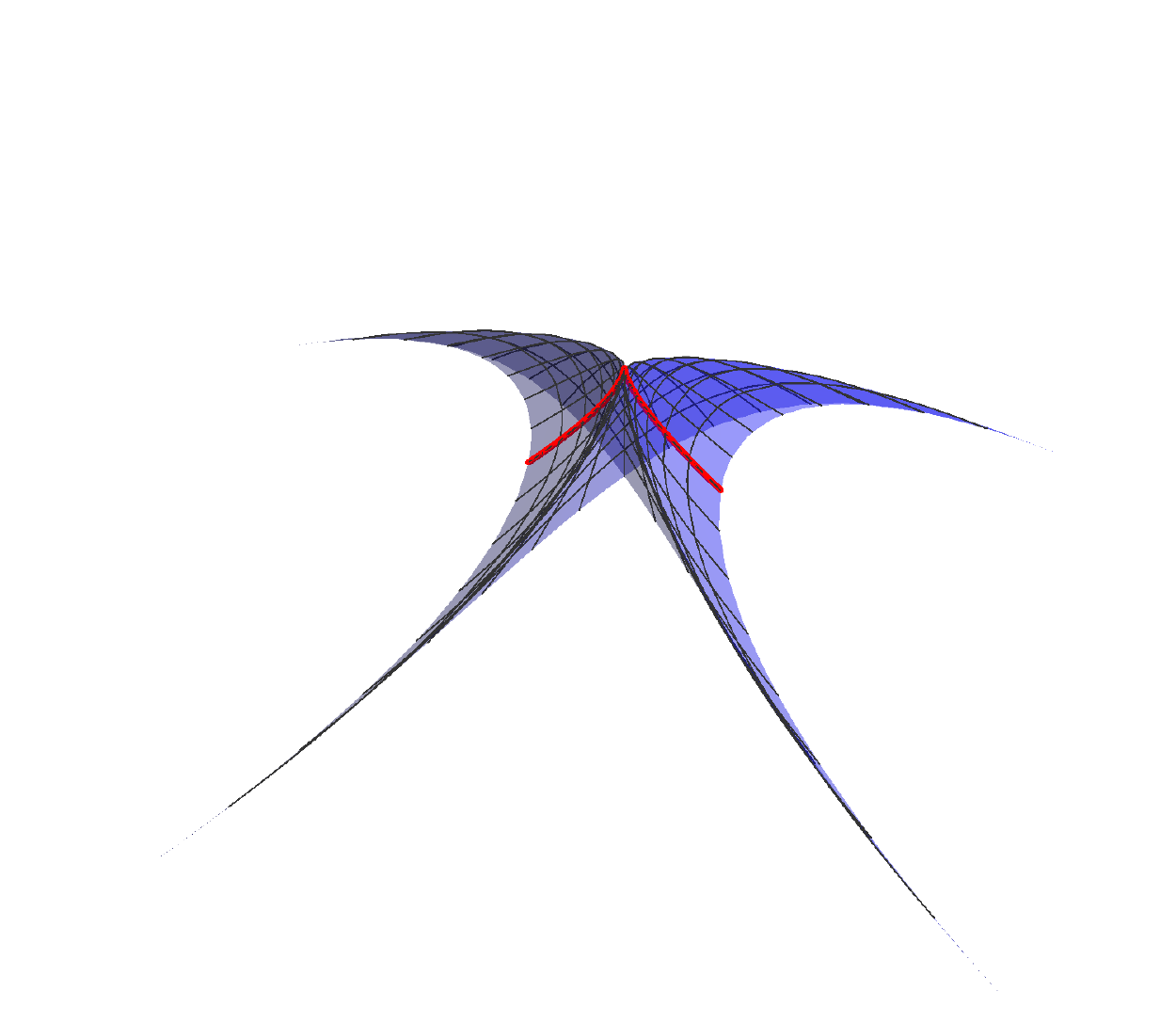}
\includegraphics[width=0.34\textwidth]{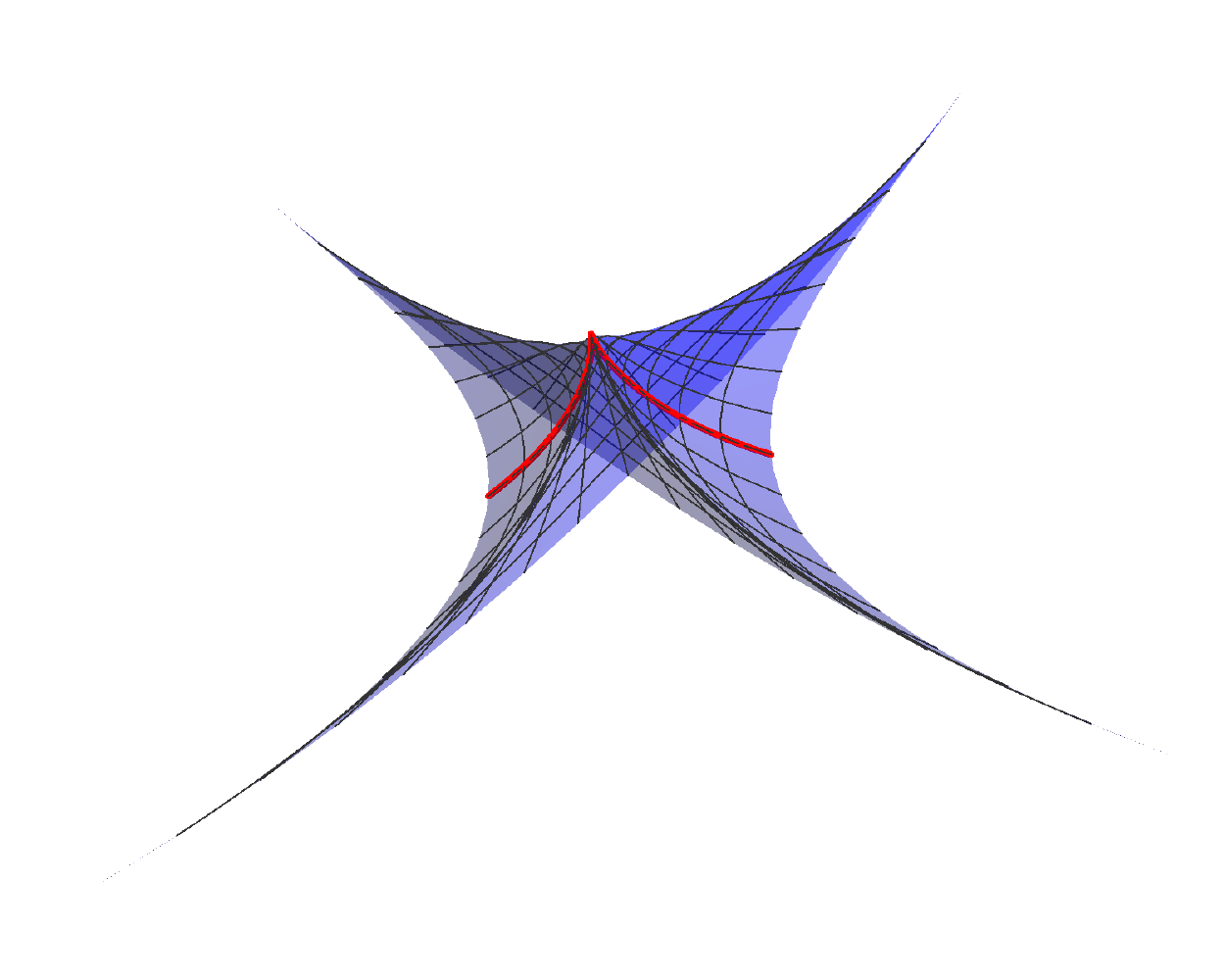}
\end{center}
\caption{Minimal surfaces in $\L^{3}$ containing the involute of the circles $x_1^2-x_3^{2}=1$, $x_3^2-x_1^{2}=1$ and $x_1^2+x_2^{2}=1$ (respectively) as planar pregeodesics.}
\label{involute}
\end{figure}

\subsection{Minimal surfaces in $\L^{3}$ obtained from others}\label{new}
We start this section observing that if ${\mathcal D}_{\psi}^{\tiny{\K}}=(L_z,P_z,h_z)$ are the Enneper data of a given spacelike (respectively, timelike) minimal immersion $\psi$ in $\L^{3}$ (defined in the simply connected domain $\Omega\subset\K$) and $f:\Omega\to\K$ is a $\K$-differentiable function so that $f(z)\overline{f(z)}\neq 0$ in $\Omega$, then $$f\,{\mathcal D}_{\psi}^{\tiny{\K}}=(f\, L_z,f\,P_z,f\, h_z)$$
 are  Enneper data of a new spacelike (respectively, timelike) minimal surface in $\L^{3}$. We note that this surface is the Enneper graph of the harmonic function $h_1:\Omega\subset\K\to\r$ defined by:
$$h_1(z)=h_1(z_0)+2\,\Re\int_{z_0}^z f(z)\,h_z(z)\, dz.$$ Also, the Enneper minimal immersion associated to $h_1$ is given by: 
\begin{equation}
\label{psi1}\psi_1=\left\{\begin{aligned}
&(L_1+\overline{P_1},h_1),\qquad \K=\C,\\
&(h_1,L_1-\overline{P_1}),\qquad \K=\L,
\end{aligned}\right.
\end{equation}
where
\begin{equation}
\label{l1p1}
L_1(z):=\int_{z_0}^z f(z)\,L_z(z)\,dz,\qquad
P_1(z):=\int_{z_0}^z f(z)\,P_z(z)\,dz,
\end{equation}
are well-defined $\K$-holomorphic functions in $\Omega$.

In the following, we will use this observation and the Enneper data of the tables given in the Sections~\ref{four} and \ref{five} to construct some examples of minimal surfaces in $\L^3$.

\begin{example}[Timelike Catalan surface of 1st kind]\label{cicloide}
We consider the Enneper data of the timelike helicoid of 1st kind (see Table~\ref{tab4}) and we choose the paracomplex fuction $f(z)=2\,\textsf{sin}\,z$, with $(u,v)$ such that $0<|\sin u|\neq |\sin v|$. Then, using \eqref{l1p1}, the timelike minimal surface obtained from the new Enneper paracomplex data:
$${\mathcal D}_{\psi}^{\tiny{\L}}=(\textsf{sin}\,z\,(1+\tau\,\mathsf{sin}\, z),\textsf{sin}\,z\,(1-\tau\,\mathsf{sin}\, z),\textsf{sin}\,z\,\textsf{cos}\,z)$$
is parametrized by:
$$\psi(u,v)=\Big(-\frac{\cos (2u)\,\cos (2v)}{2},v-\frac{\cos (2u)\,\sin(2v)}{2},2\sin u\,\sin v\Big).$$
We observe that this surface has the notable property of containing an arc of the spacelike cycloid given by $\psi(0,v)$, $v\neq 0$, as a planar pregeodesic (see Figure~\ref{fig:ciclo}). So, we call it {\it timelike Catalan surface of the first kind} and we point out that in \cite{ACM} Al\'{i}as et al. construct a spacelike Catalan surface via the B\"{o}rling problem.
\end{example}

\begin{example}[Timelike Catalan surface of 2nd kind]
In this example, we start from the Enneper data of the timelike helicoid of 3rd kind (see Table~\ref{tab4}), that are defined for all $z\in\L$, and we consider the new Enneper paracomplex data:
$${\mathcal D}_{\psi}^{\tiny{\L}}=(\tau\,\textsf{sinh}\,z\,(\mathsf{cosh}\, z+1),\tau\,\textsf{sinh}\,z\,(\mathsf{cosh}\, z-1),(\textsf{sinh}\,z)^2),$$
with $z\in\L$ such that $z\overline{z}\neq 0$. In this case, from \eqref{l1p1} we obtain the timelike surface parametrized by:
$$\psi(u,v)=\Big(\frac{\cosh (2v)\,\sinh(2u)}{2}-u,2\sinh u\,\sinh v,\frac{\cosh (2u)\,\cosh (2v)-1}{2}\Big),$$
that contains an arc of the timelike cycloid $\psi(u,0)$, $u\neq 0$, as a planar pregeodesic (see Figure~\ref{fig:ciclo}). We call it {\it timelike Catalan surface of the second kind}.
\begin{figure}[h!]
        \centering
        \begin{minipage}[c]{.38\textwidth}
          \centering
          \includegraphics[width=.85\textwidth]{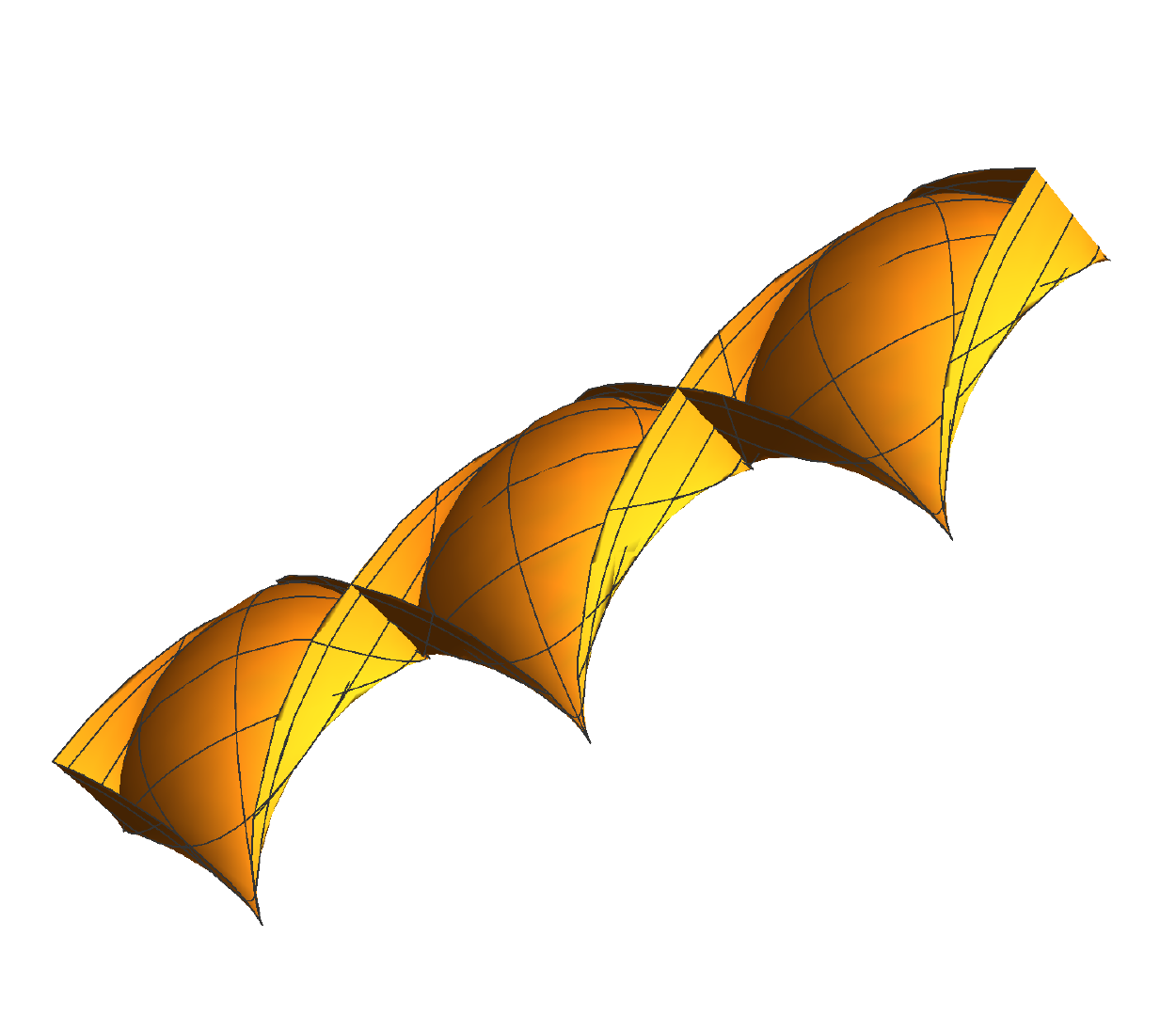}
        \end{minipage}%
        \hspace{10mm}%
        \begin{minipage}[c]{.38\textwidth}
          \centering
          \includegraphics[width=.85\textwidth]{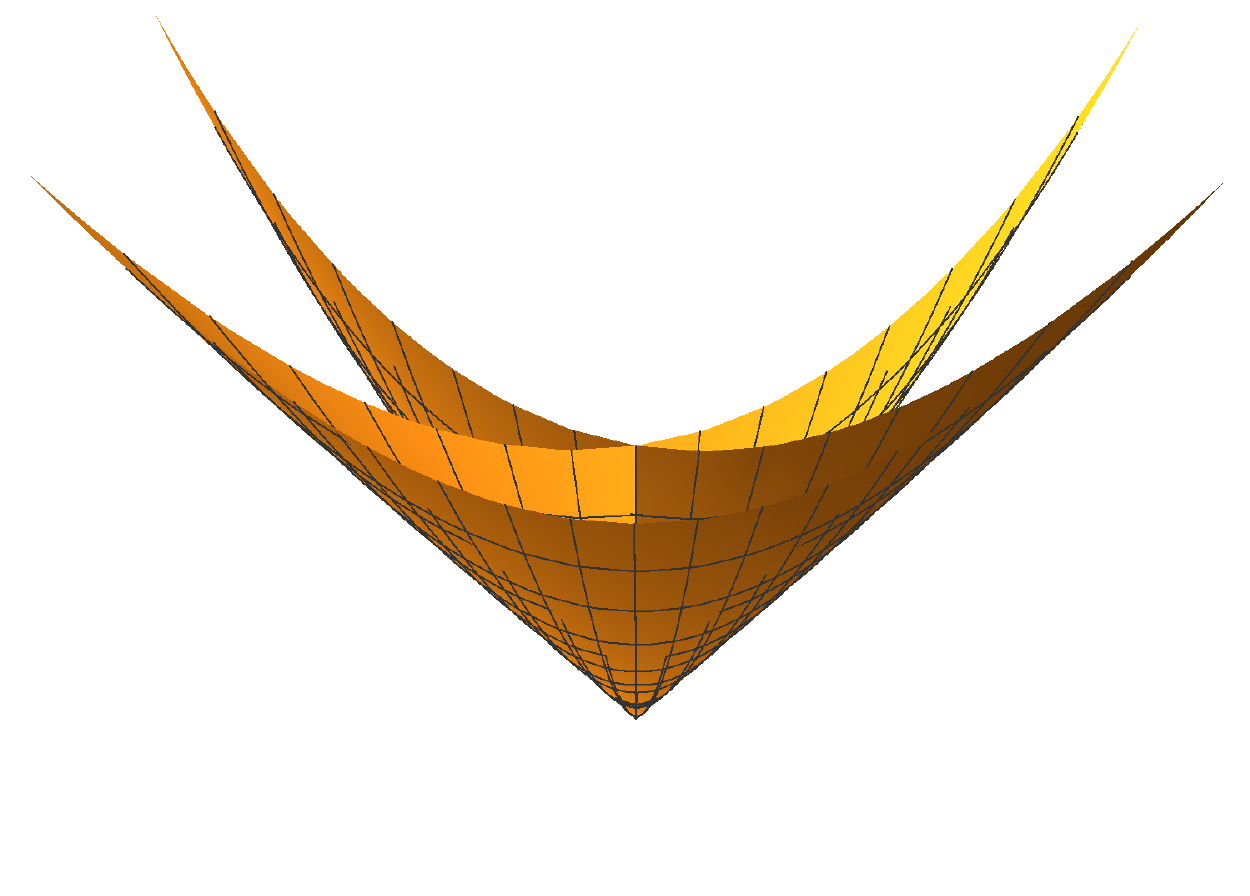}
        \end{minipage}
        \caption{\small Timelike Catalan surfaces of 1st kind and 2nd kind, respectively.\label{fig:ciclo}}
      \end{figure}
\end{example}

\begin{example}\label{altro}
Starting from the Enneper data of the spacelike hyperbolic catenoid (see the Table~\ref{tab1}) and choosing the complex function $f(z)=2\,\cos z$, with $u\in (-\pi/2,\pi/2)$, we obtain the new Enneper complex data:
$${\mathcal D}_{\psi}^{\tiny{\C}}=(\cos z\,(1+\cos z),\cos z\,(1-\cos z),-\sin z\,\cos z).$$
From \eqref{psi1} and \eqref{l1p1}, the associated spacelike minimal immersion is given by
$$\psi(u,v)=\Big(2\sin u\,\cosh v,v+\frac{\cos (2u)\,\sinh (2v)}{2},\frac{\cos (2u)\,\cosh(2v)-1}{2}\Big),$$
with $u\in (-\pi/2,\pi/2)$, and 
it intersects orthogonally the plane $x_2=0$ along the spacelike parabola $$\psi(u,0)=(2\sin u,0,-(\sin u)^2),\quad u\in (-\pi/2,\pi/2).$$ Then, this curve is a planar pregeodesic of the surface (see Figure~\ref{fig:se36}).
\begin{figure}[h!]
\centering\includegraphics[width=0.26\textwidth]{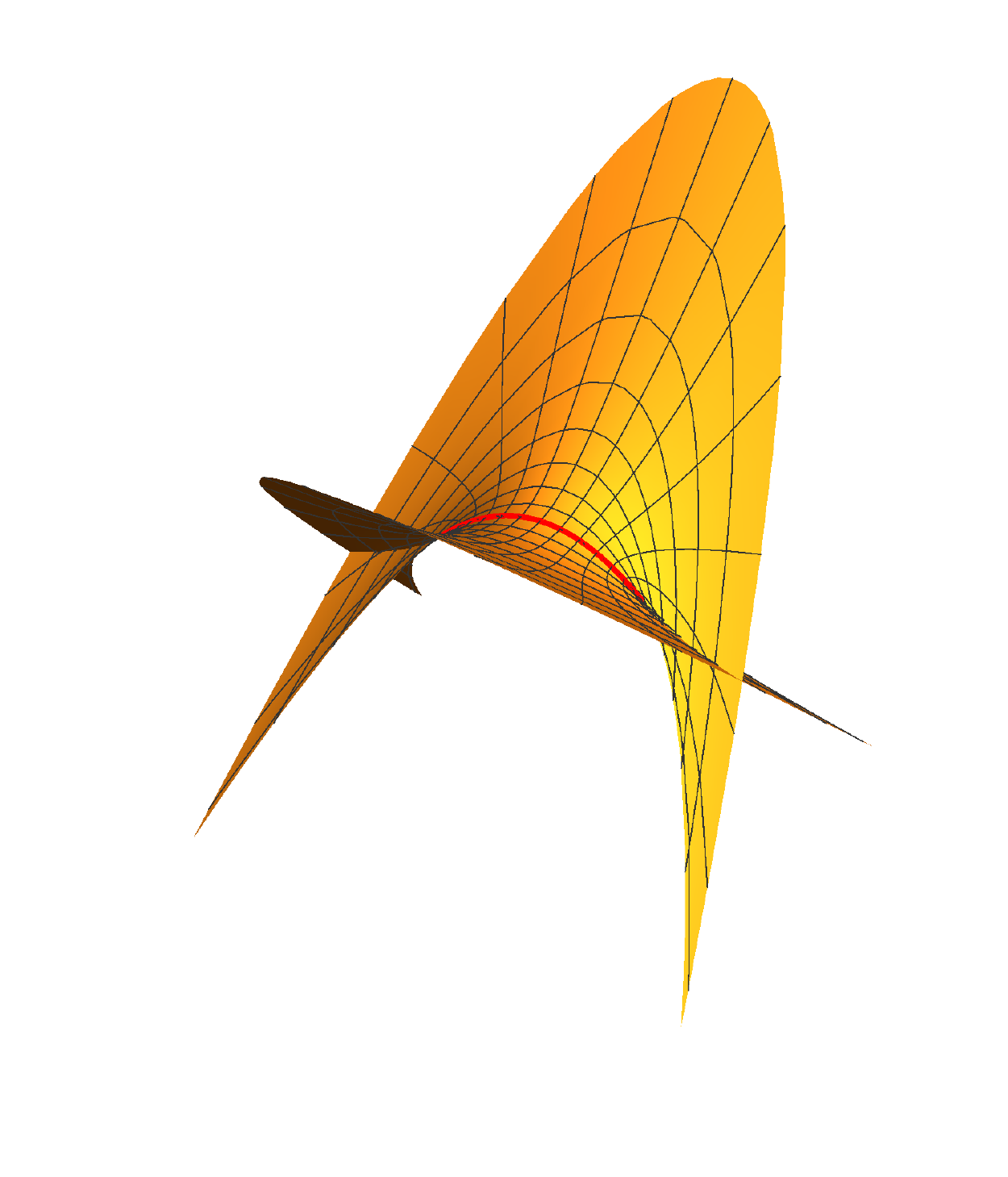}
\caption{\small Spacelike minimal surface in $\L^{3}$ containing a parabola as a pregeodesic.}\label{fig:se36}
\end{figure}
\end{example}

\subsection{A special family of minimal surfaces in $\L^3$}
Next we are going to produce a family of Lorentzian minimal surfaces in $\L^3$ whose origins are rooted in the Example~\ref{cicloide}, given in the previous section. Inspired by this example, we consider 
$n\in\mathbb{Z}$, $n>1$, and the family of paracomplex Enneper data given by:
$${\mathcal D}_{\psi_n}^{\tiny{\L}}=\Big(\textsf{sin}\, z\, (1+\tau\, \textsf{sin}\,(n z)), \textsf{sin}\, z\, (1-\tau\, \textsf{sin}\,(n z)),\textsf{sin}\, z\, \textsf{cos}\, (nz)\Big).$$
In this case, using \eqref{psi1} and \eqref{l1p1},
we obtain the following family of timelike minimal surfaces
$$\begin{aligned}\psi_n(u,v)=\Big(&\frac{\cos[ (n-1)u]\, \cos[ (n-1)v]}{n-1}-\frac{\cos[ (n+1)u]\, \cos[ (n+1)v]}{n+1},\\
&\frac{\cos[ (n-1)u]\, \sin[ (n-1)v]}{n-1}-\frac{\cos[ (n+1)u]\, \sin[ (n+1)v]}{n+1},\\&2\sin u\sin v\Big),
\end{aligned},$$
where $u\in (-\pi/4n,\pi/4n)$ and $v\in (\pi/4n,3\pi/4n)$.
Given  $n\in\mathbb{Z}$, $n>1$, we have that $\psi_n(u,v)$ is the only  minimal immersion into $\L^{3}$ containing the spacelike curve
$\alpha_n(v):=\psi_n(0,v)$, as a planar pregeodesic. If we consider the change of parameter $t=(n-1)\, v$, we have that 
$$\alpha_n(t)=\Big(\frac{\cos t}{n-1}-\frac{\cos \big(\frac{n+1}{n-1}\,t\big)}{n+1},\frac{\sin t}{n-1}-\frac{\sin \big(\frac{n+1}{n-1}\,t\big)}{n+1},0\Big),$$
that is an epycicloid  traced by a point on a circle of radius $r=1/(n+1)$ which rolls externally on a circle of radius $R=2/(n^{2}-1)$. We observe that if $n=2$, then $R=2r$, therefore  the curve $\alpha_2$ is an arc of a nephroid. Also, if $n=3$ we have that $R=r$ and, then, the curve $\alpha_3$ is an arc of a cardioid.
\begin{figure}[h!]
\begin{center}\includegraphics[width=0.23\textwidth]{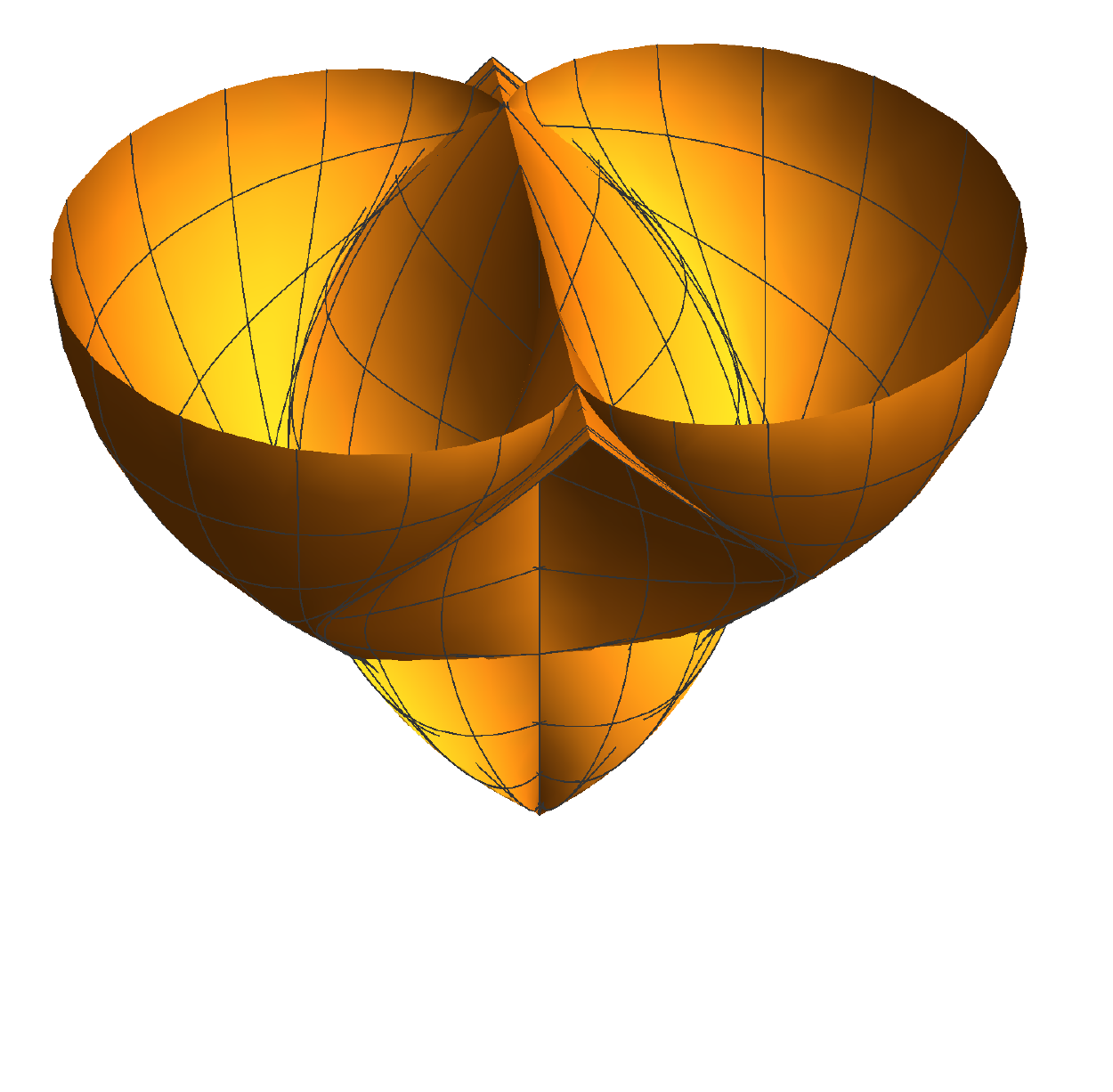}
\hspace{1.7cm}
\includegraphics[width=0.18\textwidth]{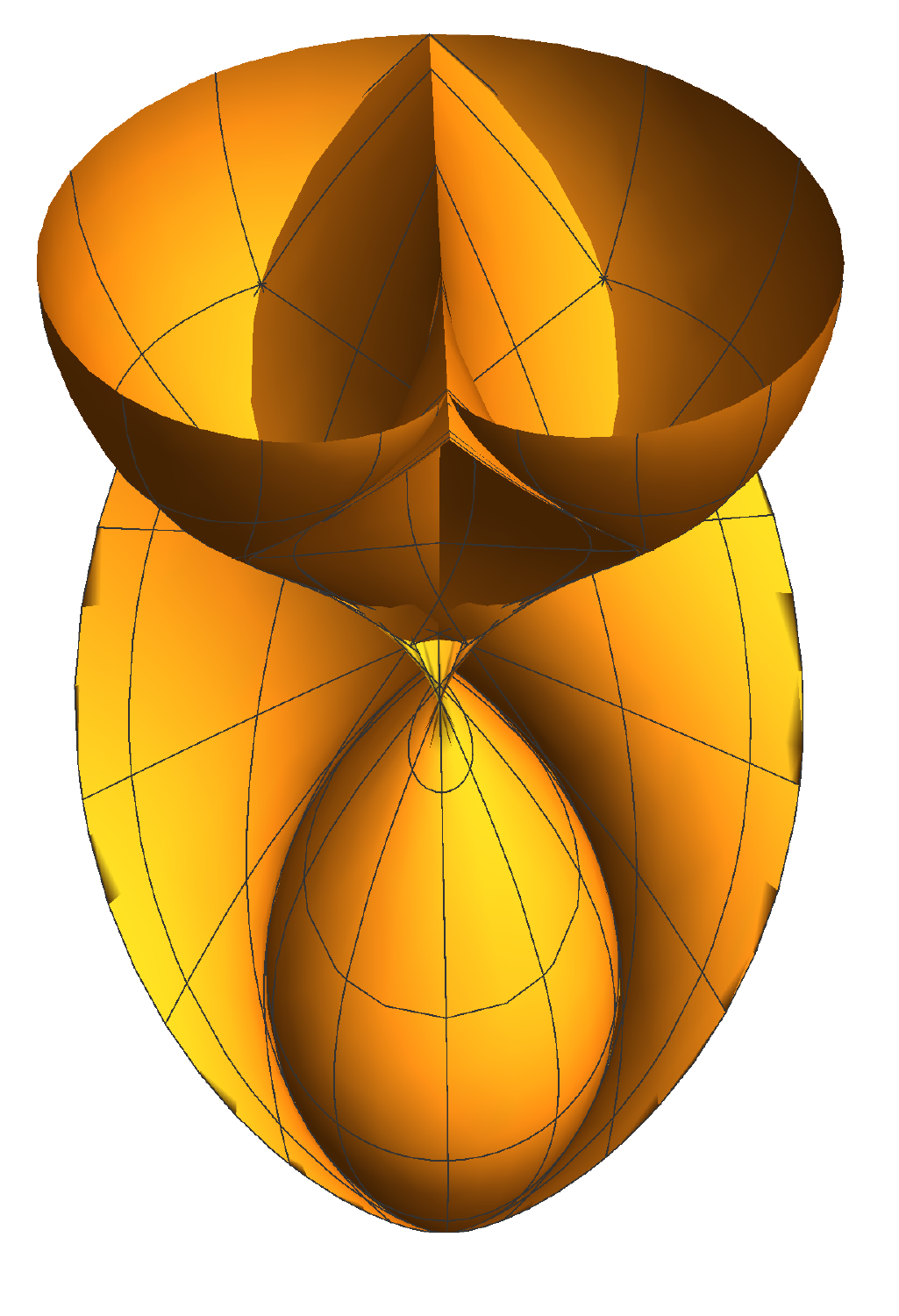}
\hspace{1.7cm}
\includegraphics[width=0.18\textwidth]{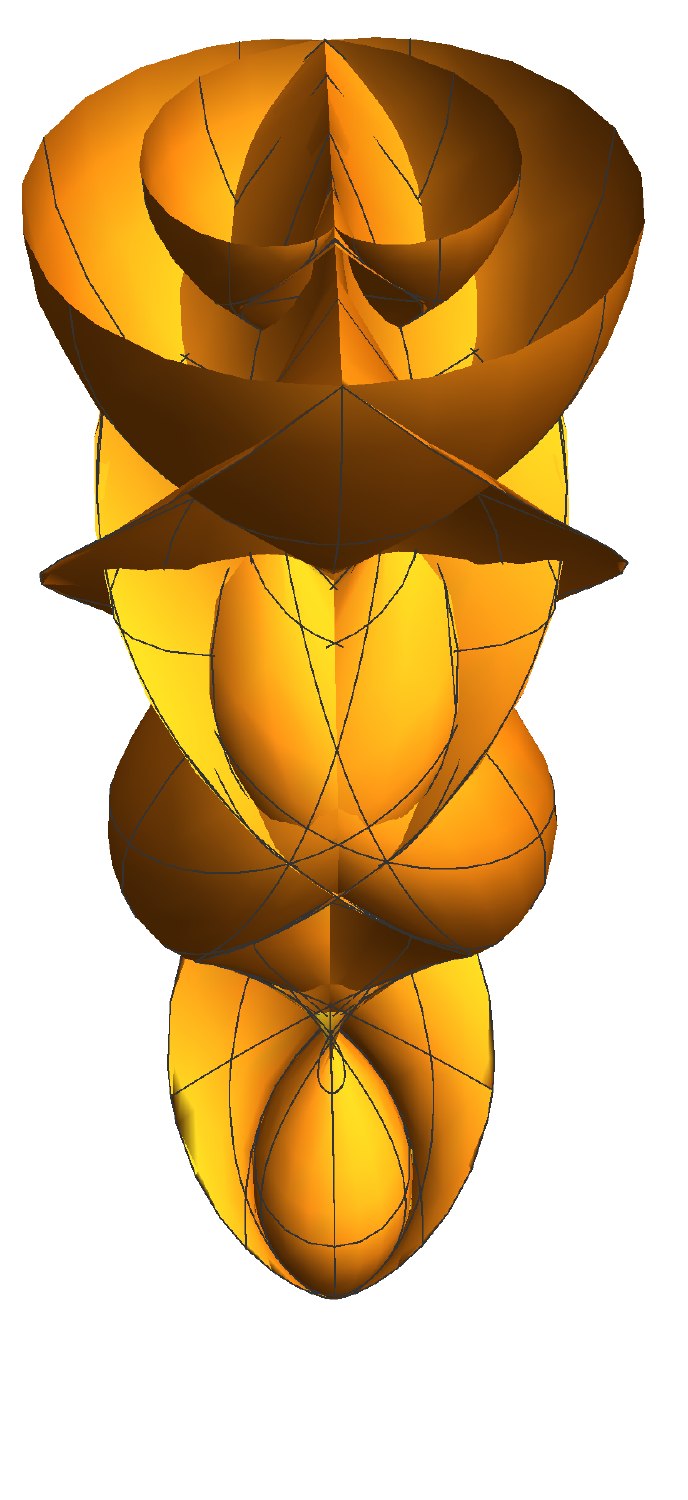}
\end{center}\caption{Timelike minimal surfaces in $\L^{3}$ containing a nephroid, a cardioid and the epicycloid for $n=5$ (respectively) as a pregeodesic.}
\label{fig:epici}
\end{figure}

\end{document}